\documentclass[11pt]{article}
\usepackage{amsmath}
\usepackage{amsfonts}
\usepackage{amscd}
\usepackage{amsmath}
\usepackage{latexsym}
\usepackage{amssymb}
\usepackage{amsthm}
\usepackage{thmtools}
\usepackage{thm-restate}
\usepackage[pdftex]{hyperref}
\hypersetup{bookmarksopen,linkcolor=blue,citecolor=magenta,colorlinks}
\usepackage{tocloft}
\usepackage[overload]{textcase}

\theoremstyle{theorem}

\theoremstyle{definition}

\newtheorem{thm}{Theorem}
\newtheorem{cor}{Corollary}
\newtheorem{lem}{Lemma}
\newtheorem{prop}{Proposition}

\newtheorem{conj}{Conjecture}
\newtheorem{defn}{Definition}
\newtheorem{rem}{Remark}
\newtheorem{exa}{Example}

\newcommand{\secref}[1]{\S\ref{#1}}

\newcommand{\corref}[1]{Corollary~\ref{#1}}

\newcommand{\field}{\mathbb{F}}
\newcommand{\FF}{\mathbb{F}}

\newcommand{\bbn}{\mathbb{N}}

\newcommand{\RR}{{\mathbb R}}
\newcommand{\bbz}{\mathbb{Z}}

\newcommand{\binomial}[2]{\genfrac{(}{)}{0pt}{}{#1}{#2}}

\renewcommand{\P}{{\mathcal P}}

\def\gb(#1){{{f_{#1,1}(x)}}}
\newcommand{\DU}[1]{\Omega(#1)}
\newcommand{\UD}[1]{{\overline{\Omega}(#1)}}
\newcommand{\MV}[1]{{f_{#1,1}}}

\newcommand{\into}{\hookrightarrow}
\newcommand{\onto}{\to \kern -7pt \to}
\newcommand{\iso}{\into \kern -5pt\onto}

\newcommand{\revonto}{\leftarrow \kern -5pt \leftarrow}

\newcommand{\Fp}{\field_p}

\newcommand{\calo}{\mathcal{O}}
\newcommand{\calp}{\mathcal{P}}
\newcommand{\calq}{\mathcal{Q}}

\newcommand{\calz}{\mathcal{Z}}
\newcommand{\skipcomment}[1]{}

\newcommand{\Gr}{{\rm Gr}}
\newcommand{\Fib}[1]{{\text{Fib}(#1)}}
\newcommand{\Fb}{{\text{Fib}}}
\newcommand{\sFb}{{\text{sFib}}}
\newcommand{\Fibpoly}[1]{{\text{Fib}_{#1}}}
\newcommand{\Q}{{\mathcal Q}}
\newcommand{\appearance}{{appearance}}
\newcommand{\Frob}{{\mathcal F}}

\newcommand{\im}{\operatorname{Im}}

\DeclareMathOperator{\GL}{GL}

\DeclareMathOperator{\Span}{span}
\newcommand{\supp}[1]{{\text{supp}({#1})}}

\newcommand{\legendre}[2]{\genfrac{(}{)}{}{}{#1}{#2}}

\newcommand{\semidirect}{\mathbin{\times \kern -1.6pt \vrule height 6pt width 0.5 pt depth 0pt}\ }
\newcommand{\set}[1]{\{#1\}}

\begin{document}

\title{Finite Planes, Zigzag Sequences, Fibonacci Numbers, Artin's Conjecture and Trinomials}
\markright{Planes in Finite Fields}

\author{H E A Campbell and D L Wehlau}
\date{\today}

\date{\today}
\begin{abstract}
We begin by considering faithful matrix representations of elementary abelian groups in prime characteristic.
The representations considered are seen to be determined up to change of bases by a single number.
Studying this number leads to a new family of polynomials which exhibit a number of special properties.
 These polynomials satisfy a three term recursion and are closely related to zigzag zero-one sequences.
 Interpreting the polynomials for the ``prime" 1 yields the classical Morgan-Voyce polynomials, which form two
orthogonal families of polynomials and which have applications in the study of electrical resistance.
 Study of the general polynomials reveals deep connections with the Fibonacci series, the order of appearance of prime numbers in the Fibonacci sequence, the order of elements in cyclic groups, Artin's conjecture
 on primitive roots and the factorization of trinomials over finite fields.
%
\end{abstract}

\maketitle

\medskip
\tableofcontents
\medskip
\newpage

\section{Introduction}
In this article we follow a mathematical path that winds through the topics listed in our title.  These seemingly unrelated  subjects are shown to have fundamental hidden connections.  We provide complete proofs of the results we present here.  Our principal tool is mathematical induction and  all of our proofs are accessible to a good undergraduate mathematics student.   The title topics are all fascinating areas of study and we are only able to provide a glimpse of each in turn.   The reader whose has not yet seen these topics will have a gentle introduction to each subject and an invitation to explore it further.

\section{Representations of $C_p \times C_p$}
This paper arose from a collaboration with Jianjun Chuai and R.~J.~Shank, \cite{Campbell+Chuai+Shank+Wehlau:17}.
Let $p$ be a prime number.  We denote the cyclic group of order $p$ by $C_p$ and let $\field$ denote the algebraic closure of the finite prime field $\field_p$.  We consider two-dimensional faithful representations of the group $C_p \times C_p = \langle \sigma_1,\sigma_2\rangle$ defined over $\field$.  Let $\rho : C_p \times C_p \to \GL_2(\field)$ be such a representation and let $\sigma_1,\sigma_2$ be fixed generators of $C_p \times C_p$.   Since $\sigma_1$ and $\sigma_2$ commute we may put
 $\rho(\sigma_1)$ and $\rho(\sigma_2)$ in upper triangular form simultaneously.  Then we have
    $$
        \rho(\sigma_i) = \begin{pmatrix} 1 & u_i\\ 0 & 1 \end{pmatrix} \quad \text{ where } u_i \in \field\ .
    $$
Since $\rho$ is faithful, $\{u_1,u_2\}$ is an ordered basis for $\calp_\rho := \Span_{\Fp}\{u_1,u_2\} = \im(\rho)$, a two dimensional $\field_p$-plane in $\field$.  Moreover any such ordered basis of a plane yields a representation of $C_p \times C_p$.

Two such representations $\rho$ and $\rho'$ are equivalent if there is a change of basis matrix $A\in \GL_2(\FF)$ such that $A \rho A^{-1} = \rho'$.  A simple computation shows that $A  \rho  A^{-1} (\sigma_i) = \rho'(\sigma_i)$ for $i=1,2$ if and only if there exists $c \in \field^*$ such that $c u_i  = u_i'$ for $i=1,2$.

Thus we see that the problem of classifying characteristic $p$ representations of $C_p \times C_p$ up to equivalence amounts to determining the orbits of two dimensional $\field_p$-subspaces under the action of $\field^*$ by dilation.

\section{Invariant theory problem and solution}  
The problem of describing all the orbits under a group action lies in the domain of invariant theory, a classical subject associated with Hilbert and the birth of modern algebra. For our purposes suppose  $V$ is an $n$ dimensional $\field$ vector space and consider an ordered basis $\{x_1,x_2,\dots,x_n\}$ for the dual space $V^* := \text{hom}(V,\field)$.  The group $\GL(V)\cong\GL(n,\field)$ acts naturally on $V$.  This induces an action on $V^*$ via $(\sigma\cdot f) (v) = f(\sigma^{-1}\cdot v)$ for all $f \in V^*$, $\sigma \in \GL(V)$ and $v \in V$.   Here $\sigma^{-1}$ is used (rather than just $\sigma$) in order to make this a left action on $V^*$.  Extending this action multiplicatively and additively we get an induced an action of $\GL(n,\field)$ on the polynomial ring $\field[x_1,x_2,\dots,x_n]$.  A polynomial $f \in \field[x_1,x_2,\dots,x_n]$ is said to be $\GL(n,\field)$-invariant if $\sigma\cdot f = f$ for all $\sigma \in \GL(V)$.  The set of invariant polynomials forms a subalgebra denoted by $\field[x_1,x_2,\dots,x_n]^{\GL(V)}$.

Following Dickson \cite{Dickson-BinaModuGrouthei:11}  we write $[e_1\, e_2\, \dots \, e_n]$ to denote the determinant of the $n \times n$ matrix whose $ij$ entry is $x_j$ raised to the power ${p^{e_i}}$ where each $e_i$ is  a non-negative integer.  Dickson proved the celebrated theorem that in this setting the ring of invariant polynomials is the polynomial ring $\field[I_0,I_1,\dots,I_{n-1}]$ where
    $$
        I_k = [0\,1\,\dots\,\widehat{k}\dots n]/[0\,1\,\dots (n-1)]\ .
    $$
Note that each $I_k$ is a polynomial and that $I_0 = [0\,1\,\dots\,(n-1)]^{p-1}$.

In this paper we will concentrate on the case $n=2$:
    $$
        \field[x,y]^{\GL(2,\field_p)} = \field[I_0,I_1]
    $$
  where
    \begin{align*}
        I_1 = [0,2]/[0,1] &= \frac{\left| \begin{matrix}  x&y \\x^{p^2} & y^{p^2}\end{matrix}\right|}
                                         {\left| \begin{matrix} x&y\\x^{p} & y^{p}\end{matrix}\right|}
                                         = \sum_{j=0}^{p} x^{j(p-1)}y^{(p-j)(p-1)}\\
        \text{ and }&\\
        I_0 = [1,2]/[0,1] &= \frac{\left| \begin{matrix}  x^p&y^p\\x^{p^2} & y^{p^2}\end{matrix} \right|}
                                                        {\left| \begin{matrix} x& y \\x^{p} & y^{p}\end{matrix}\right|}  = [0,1]^{p-1}
                                                         = \sum_{j=1}^{p} x^{j(p-1)}y^{(p+1-j)(p-1)}\ .
    \end{align*}
The fact that $I_0$ and $I_1$ are $\GL(2,\field_p)$-invariants is precisely the statement that $I_j(x,y) = I_j(ax+cy,bu+dv)$ for $j=1,0$ for any two bases $\set{x,y}$ and $\set{ax+cy,bx+dy}$ of the $\field_p$ plane spanned by $\set{x,y}$.

For $n=2$ we have $[i,j] = x^{p^i} y^{p^j} - x^{p^j} y^{p^i}$.
\begin{lem}\label{squarebrackets}
We have the following elementary properties of $[i,j]$.
    \begin{enumerate}
        \item $[i,j] = -[j,i]$;
        \item  $[i,j]^p = [i+1,j+1]$; \label{pth powers of brackets}
        \item $ [i,\ell]= [i,j] + [j,k] +[k,\ell]$ for $i < j \le k < \ell$;
        \item $[i,j][k,\ell] - [i,k][j,\ell] + [i,\ell][j,k]=0$;
        \item $[i,j]$ is bi-linear over $\field$;
        \item $[0,j]$ divides $[0,kj]$ for all positive integers $j,k$.
       \end{enumerate}
\end{lem}
Each of the above properties is easily verified.  We note that Lemma~\ref{squarebrackets}(\ref{pth powers of brackets}) implies that $I_j(x^p,y^p) = I_j(x,y)^p$ for $j = 0, 1$.

Consider a non-zero scalar $\lambda \in \field$.  Clearly  $I_1(\lambda x, \lambda y) = \lambda^{p(p-1)}I_1(x, y)$ and $I_0(\lambda x, \lambda y)  = \lambda^{(p+1)(p-1)}I_0(x, y)$.

Thus the rational function $I_1^{p+1}/I_0^{p}$ takes on the same value on $(x,y)$ as it does on $(\lambda x, \lambda y)$.

It will be convenient to introduce a sign change and use the function
     $$
        \nu := - I_1^{p+1}/I_0^{p} = -\frac{[0,2]^{p+1}[0,1]^{p}}{[0,1]^{p+1}[1,2]^p} = -\frac{[0,2][1,3]}{[0,1][2,3]} .
     $$
Note that $I_0(x,y)=0$ if and only if $\{x,y\}$ is linearly dependent over $\field_p$.  Thus $\nu$ is well defined on the set of two dimensional $\field_p$-planes contained in $\field$.  For a plane $\calp$ we will write $\nu(\calp)$ to denote $\nu(x,y)$ where $\{x,y\}$ is any basis of $\calp$.

\section{Planes in $\field_{p^m}$}\label{planes}
We let $\Gr_2(\field)$ denote the Grassmannian of two dimensional $\field_p$-planes contained in $\field$.  The Grassmannian is an important algebraic object, much studied and possessing lots of structure.  For us it will suffice to consider $\Gr_2(\field)$ as merely the set of $\field_p$-planes contained in $\field$. The group $\field^*$ acts on $\Gr_2(\field)$ and we will denote by $\calo(\calp)$ the $\field^*$-orbit of the plane $\calp \in \Gr_2(\field)$.  We note that $\nu$ is a bijection between the $\field^*$-orbits on $\Gr_2(\field)$ and $\field$.

Given a faithful representation $\rho:C_p \times C_p \to \field$ a natural question is to ask for the smallest field $\field_{p^m}$ over which a representation equivalent to $\rho$ exists.  The simplest answer is the smallest field containing $u_1/u_2$.  We can also give a more complicated answer by working with $\nu(\rho)$ as follows.

For every positive integer $m$, the field $\field_{p^m}$ of order $p^m$ is contained in $\field$.  In fact $\field=\cup_{m=1}^\infty \field_{p^m}$, and $\Gr_2(\field) = \cup_{m=1}^{\infty}\Gr_2(\field_{p^m})$ where each of the sets $\Gr_2(\field_{p^m})$ of two dimensional $\field_p$ planes in $\field_{p^m}$ is finite,   see Proposition~\ref{number_of_orbits} below.  Thus the function $\nu$ takes on only finitely many values $\calz(m) := \{z_1,z_2,\dots,z_\ell\}$ when restricted to $\Gr_2(\field_{p^m})$.

One plane in $\Gr_2(\field)$ is distinguished, namely the plane
    \begin{align*}
        \field_{p^2} &= \{x \in \field \mid x^{p^2}-x=0\} = \{x \in \field \mid [0,2](x,1)=0\} \\
            &=  \{x \in \field \setminus \field_p \mid \nu(x,1)=0\}\ .
    \end{align*}
Clearly $0 \in \calz(m)$ if and only if $\field_{p^2} \subset \field_{p^m}$ if and only if $m$ is even.  Again it will be convenient to exclude the value 0 and so we define
    \begin{align*}
       \Gr_2^\circ(\field) &= \Gr_2(\field) \setminus \calo(\field_{p^2}) \\
       &\mbox{ and } \\
    \calz^\circ(m) &= \calz(m) \setminus \{0\}\ .
    \end{align*}

We consider the polynomial having all simple roots which vanishes at the points of $\calz^\circ(m)$:
    $$
        f_{m,p}(X) := \prod_{z \in \calz^\circ(m)}(X-z) \in \field[X]\ .
    $$
Let $\calp \in \Gr_2(\field)$ have basis $\{x,y\}$ and suppose $\nu(x,y) \neq 0$.   Then the $\field^*$-orbit of $\calp$ includes a subplane of $\field_{p^m}$ if and only if $f_{p,m}(\nu(x,y)) = 0$.  In other words, the representation of $C_p \times C_p$ corresponding to the plane $\calp$
 is defined over $\field_{p^m}$ if and only if $f_{p,m}(\nu(x,y)) = 0$.  The polynomials $f_{m,p}(X)$ will play a central role in the rest of this article.

 Let $\ell(m,p)=\#\calz(m)$ denote the number of orbits of $\field_p$-planes in the field $\field_{p^m}$.  Using the following proposition we obtain explicit expressions for $\deg(f_{m,p}(X))$.

\begin{prop}\label{number_of_orbits}
The number of $\field^*_{p^m}$-orbits of planes in $\field_{p^m}$ is
    $$
       \ell({m,p}) = \begin{cases}
              \qquad \frac{p^{m-1}-1}{p^2-1}, & \mbox{if } m \mbox{ is odd};\\
             1+\frac{p^{m-1}-p}{p^2-1}, & \mbox{if } m \mbox{ is even}.\\
      \end{cases}
    $$
Hence
    $$\deg(f_{m,p}(X)) =
        \begin{cases}
          \frac{p^{m-1}-1}{p^2-1}, & \mbox{ if } m \mbox{ is odd};\\
          \frac{p^{m-1}-p}{p^2-1}, & \mbox{if } m \mbox{ is even}.
        \end{cases}
    $$
\end{prop}

\begin{proof}
Clearly if an $\field^*$-orbit meets $\Gr_2(\field_p^m)$ then its intersection with  $\Gr_2(\field_p^m)$ is an $\field_{p^m}^*$-orbit.  The number of ordered bases for a plane in $\field_{p^m}$ is given by $(p^m-1)(p^m-p)$.  Each $\field_p$-plane has $(p^2-1)(p^2-p)$ ordered bases.  Thus there are $\frac{(p^m-1)(p^m-p)}{(p^2-1)(p^2-p)}$ elements of $\Gr_2(\field_{p^m})$.  The group $\field_{p^m}^*$ acts on these planes by scalar multiplication.

For $m$ odd each plane $\P$ has stabilizer $\{\lambda \in \field^* \mid \lambda \P = \P\}$ is $\field_p^*$ and so each $\field_{p^m}$-orbit has size $\frac{p^m-1}{p-1}$.  Thus for $m$ odd there are $\frac{p^{m-1}-1}{p^2-1}$ orbits of planes in $\field_{p^m}$.

When $m$ is even the stabilizer of most planes is again $\field_p^*$.  However the plane $\field_{p^2}$ and those planes in its orbit are stabilized by the subgroup $\field_{p^2}^* \subset \field_{p^m}^*$.  Hence this orbit has size $\frac{p^m-1}{p^2-1}$.  Thus there are $\frac{(p^m-1)(p^m-p)}{(p^2-1)(p^2-p)} - \frac{p^m-1}{p^2-1}$ planes not in the orbit of $\field_{p^2}$.  Each of these planes lies in an orbit of size $\frac{p^m-1}{p-1}$ and thus there are $1+ \frac{p^{m-1}-p}{p^2-1}$ of orbits in $\field_{p^m}$ when $m$ is even.
Since
  $$
    \deg(f_{m,p}(X)) = \#\calz^\circ(m) =
        \begin{cases}
                 \ell(m,p), & \mbox{if $m$  \text{ is odd}};\\
                 \ell(m,p)-1, & \mbox{if } m \mbox{ is even}.
               \end{cases}
  $$
the second assertion of the proposition follows.
 \end{proof}

Given a plane $\calp \in \Gr_2(\field)$ with basis $\set{x,y}$, its $\field^*$-orbit contains the plane spanned by $\set{x/y,1}$, and this plane contains the line $\field_p$.  In fact it is not difficult to see that every $\field^*$-orbit of every plane in $\Gr_2^\circ(\field)$ contains exactly $p+1$ planes containing $\field_p$, thus forming a {\em pencil} of planes.  For $z \in \field^*$, we will write $\Q^z$ to denote the pencil of $p+1$ planes $\calp_1, \calp_2,\dots,\calp_{p+1}$ each of which contain $\field_p$ and satisfy $\nu(\calp_i)=z$. We write $\widehat{\calp_i} := \calp_i \setminus \field_p$ and $\widehat{\calq^z} := \calq^z \setminus \field_p = \sqcup_{i=1}^{p+1} \,\widehat{\calp_i}$.  Then $\#\widehat{\calp_i}=p^2-p$, $\# \widehat{\calq^z} = p^3-p$ and $\#\calq^z=p^3$.

\begin{exa}
Suppose $p=3$, $m=6$ and consider the planes in $\field_{729}$.  There are $11,011$ many planes inside $\field_{p^6}$ and these come in $31$ orbits.
Hence $\calz = \set{z_0,z_1, \dots, z_{30}}$.  One of these orbits is associated to $V_0 = \field_{9}$ and has size $728/8 = 91$, we have $z_0 = \nu(V_0) = 0$ and $\widehat{\calq^0} = \field_9 \setminus \field_3$.  We observe that $V_0$ is the only plane in its orbit containing $\field_p$.

The other $30$ orbits have size $728/2=364$.  Each of these orbits contains a pencil of $4$ planes containing the line $\field_p$.  We observe that $\field_{27} \subset \field_{729}$ so that any plane, say $V_1 \subset  \field_{27}$, has $z_1 = \nu(V_1) = -1$.  Thus, $\field_{27} = \calq^1$.  We also have $\#\widehat{\calq^{z_i}} = 24$ for all $1 \le i \le 30$.
\end{exa}

\begin{exa}\label{eg:collapsing}
Consider $p=3$ and $z=-1-\sqrt{-1} \in \field_9$.  The pencil of 4 planes associated to this value of $z$ lies in
$\field_{81}$.   Each of the four planes satisfies  $\widehat{\P_i} \subset \field_{81} \setminus \field_9$ and $\nu(\P_i)=z$.
For two of the planes, say $\P_1$ and $\P_2$
we find $I_0(\P_1),I_0(\P_2) \in \field_9$ with $I_0(\P_1)=z$ and $I_0(\P_2)=z^3$.
For the other two planes $I_0(\P_3)^9 = I_0(\P_4) \notin \field_9$.
%
%
\end{exa}

\section{The polynomials $f_{m,p}(X)$}
The first few of these polynomials are given by
    \begin{align*}
        f_{2,p}(X) & = 1\\
        f_{3,p}(X) & = X+1\\
        f_{4,p}(X) & = X^p + X^{p-1} + 1\\
        f_{5,p}(X) & = X^{p^2+1} + X^{p^2} + X^{p^2-p+1} + X + 1\\
        f_{6,p}(X) & = X^{p^3+p} + X^{p^3+p-1} + X^{p^3} + X^{p^3-p^2+p} + X^{p^3-p^2+p-1}
                   + X^p + X^{p-1} + 1\\
        f_{7,p}(X) & =  X^{p^4+p^2+1} + X^{p^4+p^2} + X^{p^4+p^2-p+1} + X^{p^4+1} + X^{p^4}
           + X^{p^4-p^3+p^2+1}\\
            & \quad+ X^{p^4-p^3+p^2} + X^{p^4-p^3+p^2-p+1}
           +   X^{p^2+1} + X^{p^2} + X^{p^2-p+1} + X + 1
\end{align*}
These functions exhibit a number of fascinating and surprising properties.  Some of these are given in the following theorem.   One fact worth noting is that we can express each $f_{m,p}$ as a simple expression in $p$.  A priori we see no reason why this should have been true.  Indeed it is perhaps surprising that $f_{m,p}(X)$ lies in $\field_p[X]$.

\begin{thm}\label{properties of fmp}
    \begin{enumerate}
        \item\label{recursive relation} Define  $\theta(r,p)\! :=\! p^r\! -\! p^{r-1}\! +\! \dots \!+\! (-1)^r$.  Then
            $$
                f_{m,p}(X) = X^{\theta(m-3,p)} f_{m-1,p}(X) + f_{m-2,p}(X)\ .
            $$
        \item\label{01 coeffs} All coefficents of $f_{m,p}$ are either 0 or 1, that is,
            $$
                f_{m,p}(X) = \sum_{i \in \supp{m,p}} X^i
            $$
         where
            $$
                \supp{m,p} = \{i \in \bbn \mid \text{the coefficient of }X^i \text { in } f_{m,p}\text{ is non-zero}\}\ .
            $$
        \item\label{fibb} The number of terms in $f_{m,p}(X)$ is  $f_{m,p}(1)=\Fib{m}$ where $\Fib{m}$ is the $m^{\rm th}$ Fibonacci number.
        \item\label{strong div seq} This family of polynomials for a fixed $p$ forms a strong division sequence, that is,
            $$
                \gcd(f_{m,p}(X),f_{n,p}(X)) = f_{\gcd(m,n),p}(X)\ .
            $$
        \item\label{odd binate} If $m$ is odd and $i \in \supp{m,p}$ then the base $(-p)$ expansion of $i$ involves only the digits 0 and 1.
        \item\label{even binate} If $m$ is even and $i \in \supp{m,p}$ then the base $(-p)$ expansion of $-i$ involves only the digits 0 and 1.
        \item \label{direct limit}  We have that $\supp{m,p} \subset \supp{m+2,p}$ for all $m$.
  \end{enumerate}
\end{thm}

\begin{proof}
The key property here is the recursive relation expressed in property~(\ref{recursive relation}).  We will defer the proof of property~(\ref{recursive relation}) until later.  Here we show how most of the other statements in the Proposition follow quickly from (\ref{recursive relation}) with the proofs of properties~(\ref{odd binate}) and (\ref{even binate}) following from Theorem~\ref{thm:main} below.

All the above properties are easily verified for $m=2,3$.  Hence we suppose that $m \geq 4$ and proceed by induction assuming property~(\ref{recursive relation}) is true.  Evaluating at $X=1$ the recursive relation becomes $f_{m,p}(1) = f_{m-1,p}(1) + f_{m-2,p}(1)$ which implies (\ref{fibb}) by induction.

We may derive (\ref{strong div seq}) from (\ref{recursive relation}) but we prefer to give a more geometric proof.  We know the linear factors of $f_{m,p}(X)$ are the linear polynomials $X-\nu(\calp)$ where $\calp$ is a plane in $\field_{p^m}$  that is not in the orbit of $\field_{p^2}$.  Such a factor will divide both $f_{m,p}(X)$ and $f_{n,p}(X)$ if and only if $\calp$ lies in both $\field_{p^m}$ and  $\field_{p^n}$.  But $\field_{p^m} \cap \field_{p^n} = \field_{p^g}$ where $g = \gcd(m,n)$.  Thus $\gcd(f_{m,p}(X),f_{n,p}(X)) = \prod_{\substack{\calp \subset \field_{p^g}\\ \calp \ne \field_{p^2}}} (X-\nu(\calp)) = f_{g,p}(X)$.

We have $\supp{m,p} :=  \{i \in \bbn \mid \text{the coefficient of }X^i \text { in } f_{m,p}(X)\text{ is non-zero}\}$.  It is easy to verify that  $\theta(m-3,p) =\deg(X^{\theta(m-3,p)}) > \deg(f_{m-2,p})(X)$ for $m \geq 3$.  Hence we see that the set $(\theta(m-3,p)+\supp{m-1,p}) := \{\theta(m-3,p) + i \mid i \in \supp{m-1,p}\}$ is disjoint from $\supp{m-2,p}$.  Thus (\ref{recursive relation}) implies that $\supp{m,p} = (\theta(m-3,p)+\supp{m-1,p}) \sqcup \supp{m-2,p}$.  Here $\sqcup$ denotes the union of disjoint sets.  From this,  (\ref{01 coeffs}) follows given that the only coefficients of $f_{2,p}(X)$ and $f_{3,p}(X)$ are $0$ and $1$.


Property (\ref{direct limit}) is clear from the above.

\end{proof}

\begin{defn}
In light of Theorem~\ref{properties of fmp}(\ref{recursive relation}), it is natural  to define $f_{1,p}(X)=1$ and $f_{0,p}(X)=0$.  These definitions make the above theorem hold for all $m \geq 0$.
\end{defn}

\begin{rem}  The recursive property (\ref{recursive relation}) is surprising to us as it implies that the values
of $\nu$  for planes in $\FF_{p^m}$ is determined by the values for planes in
$\FF_{p^{m-1}}$ and $\FF_{p^{m-2}}$.
\end{rem}

\begin{rem}
It may be surprising to some readers to see a negative integer used as a base.  However, there is no obstruction to using negative bases and indeed they possess the advantage that negative integers may be expressed without requiring the use of a minus sign.
\end{rem}

\section{Zigzag Sequences}
We will consider finite integer sequences.    Let
    $$
        V=(v_{n-1},v_{n-2},\dots,v_1,v_0)
    $$
be a sequence of integers.  We say the length of this sequence is $n$.  Given any integer $b$ we interpret $V$ as a base $b$ representation of a number denoted $||V||_b$, that is, $||V||_b := \sum_{i=0}^{n-1} v_i b^i$.  Note we do not constrain the values of the integers $v_i$ nor the value of the integer $b$.

Given $n \ge m$ and $V=(v_{n-1},v_{n-2},\dots,v_0)$, $W=(w_{m-1},w_{m-2},\dots,w_0)$  we write $V+W$ to denote the sequence
    $$
        V+W = (v_{n-1}+w_{n-1},v_{n-2}+w_{n-2},\dots, v_0+w_0)
    $$
where $w_{m}=w_{m+1}=\dots=w_{n-1}=0$.  Also for an integer $c$ we write $c V = (c v_{n-1},c v_{n-2},\dots,c v_1,c v_0)$.  We note that $X^{||c V + W||_b} = (X^{||V||_b})^c \cdot X^{||V||_b}$.

A  finite zero-one sequence of length $n$ is a {\it down-up} or {\it zigzag}  sequence if it is a finite sequence $(\varepsilon_{n-1},\varepsilon_{n-2},\dots,\varepsilon_1,\varepsilon_0)$ of zeros and ones such that
$$\varepsilon_{n-1}\geq \varepsilon_{n-2} \leq \varepsilon_{n-3} \geq \varepsilon_{n-4} \leq \dots \varepsilon_{0}\ .$$
  To emphasize that a sequence is a zero-one sequence we will use $\varepsilon_i$ to denote its entries.
\begin{defn}
We write $\DU{n}$ to denote the set of down-up sequences of length $n$.
\end{defn}

The first few of these sets are
    \begin{align*}
        \DU{0} &= \{(0)\} \text{ (by convention)}\\
        \DU{1} & = \{ (1), (0)\}\\
        \DU{2} & = \{ (11), (10), (00)\}\\
        \DU{3} & = \{ (111), (101), (001), (100), (000)\}\\
        \DU{4} & = \{ (1111), (1011), (0011), (1110), (1010), (0010), (1000), (0000)\}\\
     \end{align*}

Suppose that $n$ is odd and that $u=(\varepsilon_{n-1}, \varepsilon_{n-2}  \dots \varepsilon_1, \varepsilon_{0}) \in \DU{n}$.  Then $\varepsilon_1 \leq \varepsilon_0$.  Hence if $\varepsilon_0=0$ then $\varepsilon_1=0$.  From this it follows that for $n$ odd,
   \begin{align*}
        \DU{n} &= \{ (\varepsilon_{n-1}, \varepsilon_{n-2}  \dots \varepsilon_2,0,0) \mid (\varepsilon_{n-1}, \varepsilon_{n-2}  \dots \varepsilon_2) \in \DU{n-2}\}\\
        &\quad \sqcup  \{ (\varepsilon_{n-1}, \varepsilon_{n-2}  \dots \varepsilon_2,\varepsilon_1,1) \mid (\varepsilon_{n-1}, \varepsilon_{n-2}  \dots \varepsilon_2,\varepsilon_1) \in \DU{n-1}\}\ .
    \end{align*}
Similarly for $n$ even we have
    \begin{align*}
          \DU{n} &= \{ (\varepsilon_{n-1}, \varepsilon_{n-2}  \dots \varepsilon_2,1,1) \mid (\varepsilon_{n-1}, \varepsilon_{n-2}  \dots \varepsilon_2) \in \DU{n-2}\} \\
          &\quad \sqcup  \{ (\varepsilon_{n-1}, \varepsilon_{n-2}  \dots \varepsilon_2,\varepsilon_1,0) \mid (\varepsilon_{n-1}, \varepsilon_{n-2}  \dots \varepsilon_2,\varepsilon_1) \in \DU{n-1}\}\ .
    \end{align*}

From this it follows by induction that $\# \DU{n} = \Fib{n+2}$ for $n \geq 0$.
%
%

\begin{thm}\label{thm:main}
For $m \ge 2$, we have
    $$
        f_{m,p}(X) = \sum_{\varepsilon \in \Omega({m-2})} X^{(-1)^{m-1}||\varepsilon||_{-p}}\ .
    $$
\end{thm}
\begin{proof}
The proof is by induction on $m$.  Write $\omega(r)$ to denote the all ones sequence $(1,1,\dots,1)$ of length $r$.  Then $(-1)^{r}||\omega(r+1)||_{-p} = \theta(r,p)$.



For $m=2$ we have $f_{2,p}(X) = 1 = X^{-||0||_{-p}}$.

For $m=3$ we have $f_{3,p}(X) = X+1 = X^{||(1)||_{-p}} + X^{||(0)||_{-p}}$.

For $m=4$ we have $f_{4,p}(X) = X^p + X^{p-1} + 1 = X^{-||(1,0)||_{-p}} + X^{-||(1,1)||_{-p}} + X^{-||(0,0)||_{-p}}$.

Now by induction and assuming Theorem~\ref{properties of fmp}(\ref{recursive relation}) we have
    \begin{align*}
      f &_{m,p}(X) = X^{\theta(m-3,p)} f_{m-1,p}(X) + f_{m-2,p}(X)\\
        &= X^{(-1)^{m-3}||\omega(m-2)||_{-p}} \sum_{\varepsilon \in \Omega({m-3})} X^{(-1)^{m-2} ||\varepsilon||_{-p}}
            + \sum_{\varepsilon \in \Omega{(m-4)}} X^{(-1)^{m-3} ||\varepsilon||_{-p}}\\
        &= \sum_{\varepsilon \in \Omega({m-3})} X^{||(-1)^{m-3}\omega(m-2)+(-1)^{m-2}\varepsilon||_{-p}}
            + \sum_{\varepsilon \in \Omega({m-4})} X^{(-1)^{m-3}||\varepsilon||_{-p}}\\
        &= \sum_{\varepsilon \in \Omega({m-3})} X^{(-1)^{m-3}||\omega(m-2)-\varepsilon||_{-p}}
            + \sum_{\varepsilon \in \Omega({m-4})} X^{(-1)^{m-3}||\varepsilon||_{-p}}\\
        &=\sum_{\varepsilon \in (\omega(m-2)-\Omega({m-3}))} X^{(-1)^{m-3}||\varepsilon||_{-p}}
            + \sum_{\varepsilon \in \Omega({m-4})} X^{(-1)^{m-3}||\varepsilon||_{-p}}\\
        &=\sum_{\varepsilon \in (\omega(m-2)-\Omega({m-3})) \sqcup \Omega({m-4})} X^{(-1)^{m-1}||\varepsilon||_{-p}}\\
    \end{align*}
Here $\omega(m-2)-\Omega({m-3})$ denotes the set $\{\omega(m-2) - \varepsilon \mid \varepsilon \in \Omega(m-3)\}$.

It remains to show that
    $$
        (\omega(m-2)-\Omega({m-3})) \sqcup \Omega({m-4}) = \Omega(m-2)
    $$
where we are considering the elements of $\Omega({m-4})$ as the elements of $\Omega({m-2})$ whose two leftmost entries are both $0$.

Since each element of $\Omega(m-3)$ is a zero-one sequence of length $m-3$, the set $\omega(m-2)-\Omega({m-3})$ consists of zero-one sequences of length $m-2$ whose leftmost entry is a $1$.  Furthermore since each $\varepsilon \in \Omega(m-3)$ is a down/up sequence, it is easy to see that the sequence $\omega(m-2) - \varepsilon$ is also a down/up sequence.  Hence $(\omega(m-2)-\Omega({m-3})) \sqcup \Omega({m-4}) \subseteq \Omega(m-2)$.  But
    \begin{align*}
        \#\big((\omega(m-2)&-\Omega({m-3})) \sqcup \Omega({m-4})\big) = \#\Omega({m-3}) + \#\Omega({m-4})\\
            & =\Fib{m-1}+\Fib{m-2} =\Fib{m} = \#\Omega(m-2)
          \end{align*}
which completes the proof.
\end{proof}

Note that Theorem~\ref{thm:main} asserts that
  $$
    \supp{m,p} = \{(-1)^{m-1}||\varepsilon||_{-p} \mid \varepsilon \in \DU{m-2}\}\ .
  $$
From this  properties~(\ref{odd binate}) and (\ref{even binate}) of Theorem~\ref{properties of fmp} follow immediately.

In the rest of this section we will show how to complete the proof of Theorem~\ref{thm:main}.  It only remains to prove Theorem~\ref{thm:main}(\ref{recursive relation}).

Define
    \begin{align*}
        F_{1,p} &= F_{2,p} := 1\\
        F_{2k+1,p} &:= (-1)^{k}[0,2k+1] [0,1]^{-\theta(2k,p)}\quad\text{for }k\geq 1 \ \ \text{ and}\\
        F_{2k,p} & := (-1)^{k+1}\frac{[1,2]}{[0,1][0,2]}[0,2k] [0,1]^{-\theta(2k-1,p)} \quad\text{for }k\geq 1 \ .
     \end{align*}
Note that $F_{2k+1,p}(x,y)$ is defined for any pair with $\{x,y\}$ linearly independent.  Similarly $F_{2k,p}(x,y)$ is defined for any pair such that the span of $\{x,y\}$ does not lie in the $\field^*$-orbit of $\field_{p^2}$.

We prove the following proposition in the appendix.
\begin{restatable}{prop}{recursionProp}\label{recursion}
For $m \ge 3$ we have
  $$
    F_{m,p} = \nu^{\theta(m-3,p)}F_{m-1,p} + F_{m-2,p}\ .
  $$
\end{restatable}

The proof involves repeated use of the properties of Dickson's bracket polynomials. There is nothing very deep in the proof but since it is rather messy we banish it to the appendix.

We use Proposition~\ref{recursion} to prove the following.
\begin{prop}
Let $m \geq 2$ and suppose that $x,y\in \field^*$ are such that the span of $\{x,y\}$ is a plane not lying in the $\field^*$-orbit of $\field_{p^2}$.  Then
    $$
        F_{m,p}(x,y) = f_{m,p}(\nu(x,y))\ .
    $$
\end{prop}

\begin{proof}  It is a general fact from invariant theory that since $F_{m,p}$ is constant on the orbits of $\FF^*$, it must factor through
the invariant $\nu$.  We will not need to appeal to this fact here since
Proposition~\ref{recursion} shows explicitly that we can view $F_{m,p}$ as a polynomial in the variable $\nu$.
We study the properties of $F_{m,p}$ as a polynomial in $\nu$.

We have $F_{2,p}=1$ and $F_{3,p} = \nu + 1$ of degrees 0 and 1 respectively.  Using the recursive formula from Proposition~\ref{recursion} we can show that $\deg(F_{m,p}) = \ell(m,p)$ when $m$ is odd and $\deg(F_{m,p}) = \ell(m,p)-1$ when $m$ is even.  Similarly the recursive formula shows that $F_{m,p}$ is monic as a polynomial in $\nu$.  Finally  $F_{m,p}(x,y)$ vanishes on $\field_{p^m}\setminus \field_{p^2}$ since $F_{m,p}$ is divisible by  $\frac{[0,m]}{[0,1]}$ when $m$ is odd and is divisible by $\frac{[0,m]}{[0,2]}$ when $m$ is even.  Therefore $F_{m,p}(\nu)$ vanishes at all the values $\nu \in Z^\circ(m)$.

Thus as polynomials in $\nu$, both $F_{m,p}(\nu)$ and $f_{m,p}(\nu)$ are monic, share the same degree and have the same roots, and so $F_{m,p}(x,y) = f_{m,p}(\nu(x,y))$ as required.
 \end{proof}

Finally we can complete the proof of Theorem~\ref{properties of fmp} by proving property~(\ref{recursive relation}).  Suppose $m \geq 4$ and let $X=\nu(x,y)$.  Then
  \begin{align*}
        f_{m,p}(X) & = F_{m,p}(x,y) = \nu(x,y)^{\theta(m-3,p)} F_{m-1,p}(x,y) + F_{m-2,p}(x,y)\\
                        & =  \nu(x,y)^{\theta(m-3,p)} f_{m-1,p}(\nu(x,y)) + f_{m-2,p}(\nu(x,y))\\
                        & = X^{\theta(m-3,p)} f_{m-1,p}(X) + f_{m-2,p}(X).
  \end{align*}

\section{Representations of Integers}
In this section we give some simple properites of zigzag sequences in the spirit of Zeckendorf's Theorem.  Let
    $$
        \Phi(m) := \{(\varepsilon_{m-1},\varepsilon_{m-2},\dots,\varepsilon_1,\varepsilon_0) \mid
            (\varepsilon_{i},\varepsilon_{i-1}) \neq (1,1) \text{ for } i=1\dots m-1\}.
    $$
Thus $\Phi(m)$ is the set zero-one sequences of length m having no consecutive ones.  The following result is often referred to as Zeckendorf's Theorem \cite{Zeckendorf:72, ZeckendorfF:72, ZeckendorfE:72}, in the literature, although it was apparently first proved by Lekkerkerker, \cite{Lekkerkerker:52}.
\begin{thm}\label{Zeck rep}
Let $n$ be a positive integer.  Then there is a unique $m$ and a unique $\varepsilon=(\varepsilon_{m-1}, \varepsilon_{m-2}, \dots,\varepsilon_1,\varepsilon_0) \in \Phi(m)$ with $\varepsilon_{m-1}=1$ such that $\sum_{i=0}^{m-1} \varepsilon_i \cdot \Fib{i+2} = n$.
    \end{thm}
The zero-one string $\varepsilon$ guaranteed by Theorem~\ref{Zeck rep} is called the Zeckendorf representation of $n$.  Here we give similar results for zigzag sequences.

If the recursive definition of the Fibonacci numbers is extended to negative indices we get $\Fib{-1}=-1$, $\Fib{-2}=1$, $\Fib{-3}=-2, \dots$.  In general $\Fib{m} = (-1)^{m-1}\Fib{-m}$ if $m < 0$.  Then $\Fib{n}=\Fib{n-1}+\Fib{n-2}$ for all integers $n$.  These numbers , so defined, are sometimes referred to as the signed Fibonacci numbers.

Let $V=(v_{n-1},v_{n-2},\dots,v_0)$ be an integer sequence of length $n$.  We define
     \begin{align*}
       ||V||_{\Fb} &:= \sum_{i=0}^{n-1} v_i \cdot \Fib{i+1} \qquad\text{ and}\\
       ||V||_{\sFb} &:= \sum_{i=0}^{n-1} v_i \cdot \Fib{-i-2}\ .
     \end{align*}

\begin{defn}
Let $V=(v_{n-1},v_{n-2},\dots,v_0)$.  We say that $V$ is a \emph{Fibonacci representation} of $m$ if $||V||_\Fb = m$.  We say that $V$ is a \emph{signed Fibonacci representation} of $m$ if $||V||_\sFb = m$.
\end{defn}
\begin{exa}
We have
    \begin{align*}
      64 &= \Fib{10}+ \Fib{6} + \Fib{2} =55+8+1\\
      12 &= \Fib{-2}+ \Fib{-7} =-1+13 \\
      -43 &= \Fib{-2}+ \Fib{-7} + \Fib{-10}=-1+13-55 \\
    \end{align*}
\end{exa}

For each $n$ there is a natural inclusion $\iota_n : \DU{n} \hookrightarrow \DU{n+2}$ given by $\iota_n(\epsilon) = (0,0,\epsilon)$.   Clearly $||\epsilon||_\Fb = ||\iota_n(\epsilon)||_\Fb$.  We identify $\DU{n}$ with the subset $\iota_n(\DU{n})$ of $\DU{n+2}$ and define
    \begin{align*}
          \DU{\text{even}} &:= \bigcup_{i=0}^\infty \DU{2i} \quad \text{and}\\
          \DU{\text{odd}} &:= \bigcup_{i=0}^\infty \DU{2i+1} \ .
    \end{align*}

\begin{prop}\label{DUodd}
  Let $n$ be a non-negative integer.  Then there is a unique down/up zero-one sequence of odd length $\varepsilon \in \DU{\text{odd}}$ such that $||\varepsilon||_\Fb = n$.  Furthermore there is a unique down/up zero-one sequence of even length $\varepsilon \in \DU{\text{even}}$ such that $||\varepsilon||_\Fb = n$.
\end{prop}
\begin{proof}
We will show that $||\cdot||_\Fb$ gives a bijection for all $n$ between $\DU{n}$ and the interval of integers $[0,\Fib{n+2}-1]=[0,\Fib{n+2})$.  We proceed by induction treating the even and odd cases in parallel.  For $n=0$ we have $\DU{1} = \{(0)\}$ and $\{||\varepsilon||_\Fb \mid \varepsilon \in \DU{0}\} = \{0\} = [0, \Fib{2})$.  For $n=1$ we have $\DU{1} = \{(0),(1)\}$ and $\{||\varepsilon||_\Fb \mid \varepsilon \in \DU{1}\} = \{0,1\} = [0, \Fib{3})$.

Suppose then by induction that there is such a bjection for $n-2$.  Clearly $\DU{n}$ decomposes as a disjoint union:
    $$
        \DU{n} = \DU{n}^{00} \sqcup \DU{n}^{10} \sqcup \DU{n}^{11}
    $$
where $\DU{n}^{ij} = \{\epsilon \in \DU{n} \mid \epsilon_{n-1}=i, \epsilon_{n-2}=j\}$.  Clearly $\DU{n}^{00} = \iota(\DU{n-2}) = \DU{n-2}$ under our identification.  Also $\DU{n}^{10} = \{ (1,0,\epsilon) \mid \epsilon \in \DU{n-2}\}$ and  $\DU{n}^{11} \subsetneq \{ (1,1,\epsilon) \mid \epsilon \in \DU{n-2}\}$.

Thus by the induction hypothesis $\{||\epsilon||_\Fb \mid \epsilon \in \DU{n}^{00}\} = [0,\Fib{n})$ and $\{||\epsilon||_\Fb \mid \epsilon \in \DU{n}^{10}\} = \Fib{n} + [0,\Fib{n}) = [\Fib{n},2\Fib{n})$.  Clearly $||\cdot||_\Fb$ is injective when restricted to  $\DU{n}^{11}$.  We consider the image of this restriction.  This is $\{||\cdot||_\Fb \mid \epsilon \in \DU{n}^{11}\}$ which is contained in $\Fib{n}+\Fib{n-1} + [0,\Fib{n}) = [\Fib{n+1},\Fib{n+2})$.
Note that if $\epsilon\in\DU{n}^{11}$ then $\epsilon_{n-3}=1$ and thus $|||\epsilon||_\Fb \geq \Fib{n}+\Fib{n-1}+\Fib{n-2} = 2\Fib{n}$ for all $\epsilon \in \DU{n}^{11}$.

Therefore we have three injective maps
    \begin{align*}
        ||\cdot||_\Fb &: \DU{n}^{00} \to [0,\Fib{n})\\
        ||\cdot||_\Fb &: \DU{n}^{10} \to [\Fib{n},2\Fib{n})\\
        ||\cdot||_\Fb &: \DU{n}^{11} \to [2\Fib{n},\Fib{n+2})\\
        &\text{which combine to give the injection}\\
        ||\cdot||_\Fb &: \DU{n} \to [0,\Fib{n+2})
    \end{align*}
    Finally $\#\DU{n} = \Fib{n+2} = \# [0,\Fib{n+2})$ and so this map is a bijection.
\end{proof}

The proofs of the following three propositions in this section are quite similar to the proof of Proposition~\ref{DUodd} and so we omit them.

\begin{prop}\label{DUboth}
Let $n$ be any integer.  Then there is a unique down/up zero-one sequence $\varepsilon \in \DU{\text{odd}} \cup \DU{\text{even}}$ such that $||\varepsilon||_\sFb = n$.
\end{prop}
Note that in this proposition we identify the all zeros string of even and odd length so that there is a unique signed Fibonacci representation of the integer 0.

We may define up/down zero-one sequences in the obvious imitation of down/up sequences: $\epsilon = (\varepsilon_{m-1},\varepsilon_{m-2},\dots,\varepsilon_1,\varepsilon_0)$ is up/down if $\epsilon_{m-1} \leq \epsilon_{n-2} \geq \epsilon_{n-3} \leq \dots \epsilon_{0}$.  We define $\overline{\Omega}(m)$ to be the set of all up/down zero-one sequences of length $m$.  There is a family of inclusions $\overline{\iota}_m : \overline{\Omega}(m) \hookrightarrow \overline{\Omega}(m+2)$ satisfying $||\epsilon||_\Fb = ||\overline{\iota}(\epsilon)||_\Fb$.  Note that $\overline{\iota}_m$ is more complicated to define than $\iota_m$.  Using these inclusions we define $\overline{\Omega}(\text{odd})$ and $\overline{\Omega}(\text{even})$.  Then we have

\begin{prop}\label{UDodd}
Let $n$ be a non-negative integer.  Then there is a unique up/down zero-one sequence of odd length $\varepsilon \in \overline{\Omega}({\text{odd}})$ such that $||\varepsilon||_\Fb = n$.  Furthermore there is a unique up/down zero-one sequence of even length $\varepsilon \in \overline{\Omega}({\text{even}})$ such that $||\varepsilon||_\Fb = n$.
\end{prop}

Moreover we have
\begin{prop}\label{UDboth}
Let $n$ be any integer.  Then there is a unique up/down zero-one sequence of odd length $\varepsilon \in \UD{\text{odd}}$ such that $||\varepsilon||_\sFb = n$. Furthermore there is a unique up/down zero-one sequence of even length $\varepsilon \in \UD{\text{even}}$ such that $||\varepsilon||_\sFb = n$.
\end{prop}

\begin{rem}
We may consider the elements of $\DU{m}$ as lattice points in $\RR^m$.  Then their convex hull is a polytope contained in the unit $m$-cube.  We would have christened these {\em ZZ-topes} but they have already been studied by Richard Stanley, \cite{Stanley-Alternating:10}, as the   \emph{zigzag order polytopes.}  Stanley showed that this polytope has a number of  special properties.  The elements of $\DU{m}$ are the vertices of this polytope and it has no interior lattice points.  Its normalized volume is  $A_m$
the $m^{\rm th}$ zigzag number, also known as an Euler number.
 \end{rem}

\section{At the prime 1: Morgan-Voyce Polynomials}
In this section we consider the polynomials $f_{m,1}(X)$.  We view these as elements of $\bbz[X]$.  We can take either the recursive relation of Theorem~\ref{properties of fmp}(\ref{recursive relation}) or the  summation formula of Theorem~\ref{thm:main} for the formal definition of $f_{m,1}(X) \in \bbz[X]$.
%
%
Since ${\theta(2k,1)}=1$ and ${\theta(2k+1,1)}=0$ we have
    $$
        f_{m,1}(X) =
            \begin{cases}
                                   0,& \mbox{if } m =0;\\
                                   1,& \mbox{if } m =1;\\
                f_{m-1,1}(X) +  f_{m-2,1}(X),& \mbox{if } m\geq3 \mbox{ is odd};\\
                X f_{m-1,1}(X) +  f_{m-2,1}(X),& \mbox{if } m\geq 2 \mbox{ is even}.
            \end{cases}
    $$
We may separate the even and cases as follows: for $m$ even these are defined by the recursion
         $$
            f_{m,1}(X)  =
                 \begin{cases}
                     0,                                                & \mbox{if } m=0;\\
                    1,                                                & \mbox{if } m=2;\\
                    (X+2) f_{m-2,1}(X) - f_{m-4,1}(X),     & \mbox{if } m \geq 4.
                \end{cases}
        $$

 For $n$ odd these are defined by the recursion
     $$
        f_{m,1}(X) =
             \begin{cases}
                     1,                                             & \mbox{if } m=1\\
                     X+1,                                          & \mbox{if } m=3;\\
                    (X+2)f_{m-2,1}(X) - f_{m-4,1}(X), & \mbox{if } m \geq 5.
             \end{cases}
        $$

These two families of polynomials were studied by \hbox{A.M.~Morgan-Voyce} in 1959 in order to study electrical resistance in ladder networks, \cite{Morgan-Voyce:59},  see also Ferri \cite{Ferri:97}, and Swamy \cite{Swamy:00, Swamy:68, Swamy:66}.   Morgan-Voyce and others use the notation $b_k=f_{2k+1,1}$ and $B_k=f_{2k+2,1}$ to denote these polynomials.  Each of these two families is an orthonormal family of polynomials.

Non-recursive descriptions are
\begin{align*}
  f_{2k+1,1}(X) = b_{k}(X) &= \sum_{i=0}^k \binomial{k+i}{k-i}X^i&\text{ for } k \geq 0\\
  f_{2k+2,1}(X) = B_k(X) &= \sum_{i=0}^k \binomial{k+i+1}{k-i}X^i&\text{ for } k \geq 0.
\end{align*}

The Morgan-Voyce polynomials are closely related to Fibonacci polynomials.  The Fibonacci polynomials are defined by the recursion
    $$
        \Fibpoly{m}(X) =
              \begin{cases}
                  0,                                             & \mbox{if } m=0\\
                  1,                                             & \mbox{if } m=1;\\
                  X\, \Fibpoly{m-1}(X) + \Fibpoly{m-2}(X), & \mbox{if } m \geq 2.
            \end{cases}
     $$
It is easy to see that the polynomial $\Fibpoly{m}(X)$ is an even polynomial if $m$ is odd and is an odd polynomial if $m$ is even.  Thus $\Fibpoly{2k+1}(X)$ and $\Fibpoly{2k}(X)/X$ are both even polynomials.  In fact
    \begin{align*}
        \Fibpoly{2k+1}(X) &= b_k(X^2)\\
        \Fibpoly{2k+2}(X) &= X B_k(X^2)
    \end{align*}
Like both families of Morgan-Voyce polynomials, the Fibonacci polynomials also form an orthonormal family.

\section{Order of Appearance}
Recall from \secref{planes} that $\Q^z$ denotes the pencil of $p+1$ planes $\calp_1, \calp_2,\dots,\calp_{p+1}$ each of which contain $\field_p$ and satisfy $\nu(\calp_i)=z \in \field$.

Let $\field_{p^m}$ be the smallest field containing a plane $\calp$ such that $\nu(\calp)=z$.  Note that $\calp = c_{i}\cdot \calp_i$ for some $c_{i} \in \calp \subset \field_{p^m}$.  Thus each plane $\calp_i = c_i^{-1} \calp$ lies in $\field_{p^m}$ but not in any proper subfield, that is, $m$ is the least integer satisfying  $f_{m,p}(z)=0$.  Furthermore any point $x \in \widehat{\calp_i}$ yields a basis $\{x,1\}$ of $\calp_i$ and thus $\nu(x,1) = \nu(\calp_i) = z$.  In particular, no point of $\widehat{\calq^z}$ lies in any subfield of $\field_{p^m}$.

The special case of $\calq^1$ is interesting.  Since $f_{m,p}(1) = \Fib{m}$ the least $m$ such that $f_{m,p}(1) = 0$ is the least $m$ such that $p$ divides $\Fib{m}$.  This number $m$ has been studied and is called the order (or rank) of {\em appearance} of $p$ (in the Fibonacci sequence).  It is also known as the order or rank of {\em apparition} of $p$ and as the Fibonacci {\em entry point} of $p$.  We will follow the convention of using $\alpha(p)$ to denote the  order of appearance of $p$.  For example, $\alpha(11) =10$ since the prime $11$ divides $\Fib{10} = 55$  and does not divide $\Fib{i}$ for $0 \le i <10$.

In 1960, Wall \cite{Wall:60} showed that for every integer $Z$ there is a Fibonacci number $\Fib{m}$ which is divisible by $Z$.  The least such $m$ is the order of \appearance\ of $Z$ and $\Fib{n}$ is divisible by $Z$ if and only if $n$ is divisible by $m$.  Sall\'e \cite{Salle:75} showed that $\alpha(Z) \leq 2Z$ and this bound is sharp with equality if and only if $Z=6(5^j)$ where $j$ is a non-negative integer. For a prime $p$ it is known that $\alpha(p)$ divides $p - \legendre{5}{p}$  where $ \legendre{5}{p}$ denotes the Legendre symbol.  In 1913, Carmichael \cite{Carmichael:13} proved that for every $m \neq 1,2,6,12$ there exists a prime $p$ such that $\alpha(p)=m$. There is an extensive literature on the order of appearance, the Fibonacci sequence in general, and related topics.  We refer the interested reader to Kla\v{s}ka, \cite{Klaska:18}, or to the book of Koshy \cite{Koshy:11}.

In a reflection of the Fibonacci order of appearance we define the order of appearance of any element $z \in \field$ to be the least integer $m$ such that $f_{m,p}(z)=0$.  We denote this number by $\alpha(z,p)$.  In this terminology, the classical Fibonacci order of appearance of $p$ is the order of appearance of the element $1 \in\field_p \subset \field$, i.e., $\alpha(p)=\alpha(1,p)$.

\begin{exa}
   The polynomial $f_{3,p}(X)=X+1$ is the only one of these polynomials which is independent of $p$ and non-constant.
     Since $f_{3,p}(-1)=0$ we see that $\alpha(-1,p)=3$ for all primes $p$.   In fact this shows that $\Q^{-1} = \field_{p^3}$ for all $p$.
     We can also deduce this result directly from the definition of $\nu$ and the fact that $a^{p^3}=a$ for all
     $a \in \field_{p^3} \setminus \field_p$.

  We also note that applying the Rational Root Theorem to $f_m,p(X)$ shows that -1 is the only
  rational number which yields a root $f_{m,p}$ for a fixed $m$ and all $p$.    Below we will see that $f_{p,p}(-4)=0$ for all $p$.
\end{exa}

When we restrict the invariants $I_0$ and $I_1$ to pairs of the form $\{x,1\}$ they become algebraically dependent:
    \begin{align*}
        I_1(x,1) &= \frac{[0,2]}{[0,1]}(x,1) = \frac{x^{p^2}-x}{x^p-1} \text{ and }
        I_0(x,1) = \frac{[1,2]}{[0,1]}(x,1) = \frac{x^{p^2}-x^p}{x^p-1} \ .
  \end{align*}
Thus $I_1(x,1) = I_0(x,1) + 1$.  Therefore
    $$
        \nu(x,1) = -\frac{I_1(x,1)^{p+1}}{I_0(x,1)^p} = - \frac{(I_0(x,1)+1)^{p+1}}{I_0(x,1)^p}\ .
    $$

Define $\Gamma:\field \to \field$ by $\Gamma(X) = -(X+1)^{p+1}/X^p$.  Then the function $\nu(\cdot,1)$ factors as
    $$
        \nu (\cdot,1): \field \setminus \field_p \mathrel{\mathop{\longrightarrow}^{I_0(\cdot,1)}} \field  \mathrel{\mathop{\longrightarrow}^{\Gamma}} \field \ .
    $$
Clearly $\Gamma(X) = z$ if and only if  $X^{p+1} + (1+z)X^p + X +1=0$.  We define
    $$
        \gamma_z(X) := X^{p+1} + (1+z)X^p + X +1\ .
    $$
Thus $\nu(x) = z$ if and only if $\gamma_z(I_0(x,1)) =0$.  Since $\frac{d\, \gamma_z}{d\, X} = X^p +1 =(X+1)^p$, we see that $\gamma_z$ has a repeated root only when $z=0$.   If $z=0$ then $\gamma_z(0)=X^{p+1}+X^p+X+1 = (X+1)^{p+1}$ and the pencil $\calq^0$ is just the field $\field_{p^2}$ and the pencil of $p+1$ planes all coincide.

 For $z\neq 0$ the polynomial $\gamma_z(X)$ has $p+1$ distinct simple roots.  We label the roots of $\gamma_z(X)$ by $t_1,t_2,\dots,t_{p+1}$ and the
 planes $\calp_i$ of $\calq^z$ such that $I_0(\calp_i) = t_i$ for $i=1,2,\dots, p+1$.

We have $\nu(\calp_i)=z$ for each plane $\calp_i$ in the pencil $\calq^z$. We define $\tau_z(X) := \gamma_z(I_0(X,1)) = \gamma_z( (X^p-X)^{p-1} ) \in\field_p[X]$.  Then
    $$
        \{x \in \field \mid \tau(x)=0 \} = \cup_{i=1}^{p+1} \widehat{\calp}_i = \widehat{\calq}^z\ .
    $$
The splitting field for $\tau_z(X)$ is the smallest field containing $\calq^z$, and this field is $\field_{p^m}$ where $m$ is the order of appearance of $z$, by definition.

The element $z$ lies in $\field_{p^k}$ for some minimal $k$ and $\gamma_z(X) \in \field_{p^k}[X]$.  Versions of most of the following results for general values of $k$ are true but the analysis is significantly more involved than the case $k=1$, and this is the case we believe to be the most interesting for most readers.

For the remainder of this article, then, we will assume that $z \in \field_p$.

The Galois group of $\field_{p^{m}}$ over $\field_{p}$ is cyclic of order $m$ and is generated by the Frobenius map $\Frob$ given by $\Frob(X)=X^p$.  This map fixes $z$ and permutes the roots of $\gamma_z(X)$.  It also acts on the roots of $\tau_z(X)$ and thus on $\calq^z$.

Take $x_i \in \widehat{\calp_i}$ so that $\{x_i,1\}$ is a basis of $\P_i$.    Note that the plane $\calp_i$ is the unique plane in $I_0^{-1}(t_i)$ which contains the line $\field_p$.  From this fact, it is easy to see that $\Frob(\calp_i) = \calp_j$ where $\Frob(t_i) = t_j$ and $\{x_i^p,1\}$ is a basis of  $\calp_j$.  This shows that the permutation action of $\Frob$ on the set of planes $\{\calp_1,\calp_2,\dots,\calp_{p+1}\}$ is the same as its action on $\{t_1,t_2,\dots,t_{p+1}\}$ the set of roots of $\gamma_z(X)$. In particular, $\Frob(\calp_i)=\calp_i$ if and only if $t_i \in \field_p$.  We note that in this case $\calp_i$ is stabilized by $\Frob$ but it is never fixed pointwise.

The roots of $\gamma_z(X)$ fixed by $\Frob$ are those lying in $\field_p$, that is, those corresponding to linear factors in the factorization of $\gamma_z(X)$ into irreducible polynomials over $\field_p[X]$.  We may easily determine these linear factors for they are the linear divisors of $\gcd(\gamma_z(X),X^p-X)$.  Since
    $$
        \gamma_z(X) = (X+z+1)(X^p-X)\  \ +\ \ X^2 + (z+2)X +1
    $$
we see that
   \begin{align*}
       \gcd(\gamma_z&(X),X^p-X) = \gcd(X^p-X,X^2 + (z+2)X +1) \\
        = & \begin{cases}
                (X-r_+)(X-r_{-})\\
                \qquad\text{where }r_\pm = \frac{-(z+2)\pm\sqrt{z^2+4z}}{2}, &\text{if }z^2+4z \text{ is a square} \pmod{p};\\
                    1, &\text{if } z^2+4z \text{ is a non-square} \pmod{p};\\
                X-1, &\text{if } z=-4;\\
                X+1, &\text{if }z=0.
   \end{cases}
 \end{align*}
In particular, if $z\neq0,-4$ then $\gamma_z(X)$ has either two distinct linear factors or has no linear factors.

Since the roots of linear factors of $\gamma_z(X)$ behave differently than the roots of other factors we define
    $$
        \overline\gamma_z(X) :=
            \begin{cases}
                \gamma_z(X)/(X^2+(z+2)X+1),  &z^2+4z \text{ is a square} \pmod{p};\\
                \gamma_z(X), & \text{if } z^2+4z \text{ is a non-square} \pmod{p};\\
                \gamma_z(X)/(X-1), & \text{if } z=-4;\\
                1,                           &\text{if }z=0.
   \end{cases}
   $$
See Example \ref{p=19}.
\medskip

The following surprising result is very useful.
\begin{thm}\label{thm:Frob2}
Let $t$ be a root of $\overline\gamma_z(X)$ where $z \in \FF_p$.   Then $I_0(t) = \Frob^2(t)$.  In particular  $I_0(t)$ is also a root of $\overline\gamma_z(X)$.
\end{thm}
\begin{proof}
   Put $s=t+z+1$.  Then $s$ is a root of $\gamma_z(X-z-1)=X^{p+1} - z X - z$ and so $s^{p+1} = sz + z$.  Thus $s^p  = z+zs^{-1}$ and $s^{-p} = \frac{s}{zs+z}$.   Furthermore  both $s,t \notin \field_p$ since the minimal polynomial of $t$ is not linear.  Thus
   $s^p-s\neq 0$.
   Therefore
    \begin{align*}
          I_0(t) &=I_0(s) = (s^p-s)^{p-1} =  \frac{(s^p-s)^p}{s^p-s} = \frac{(zs^{-1}+z-s)^p}{zs^{-1}+z-s}\\
                      &=\frac{zs^{-p}+z-s^p}{zs^{-1}+z-s}= \frac{\frac{s}{s+1}+z-(zs^{-1}+z)}{zs^{-1}+z-s} = \frac{\frac{s}{s+1}-zs^{-1}}{zs^{-1}+z-s}\\
                      &=\frac{s-z-zs^{-1}}{(zs^{-1} + z - s)(s+1)} = \frac{-1}{s+1}
    \end{align*}
 and
    \begin{align*}
          \Frob^2(t) &= (s-z-1)^{p^2} =  (s^p-z-1)^p = (zs^{-1}+z-z-1)^p\\
                      &=(zs^{-1}-1)^p
                       =  zs^{-p}-1 = \frac{s}{s+1}-\frac{s+1}{s+1} = \frac{-1}{s+1}\ .
    \end{align*}
 Since $\Frob$ permutes the roots of $\overline\gamma_z(X)\in \field_p[X]$ the final statement of the theorem holds.
\end{proof}


\begin{thm}\label{thm:main2}
Suppose $z \in \field_p^*$  and let $m$ denote the order of appearance of $z$.   Then all irreducible factors of $\overline\gamma_z(X)$ have degree $m$.
In particular, $\gamma_z(X)$ and $\tau_z(X)$ share the same splitting field.  Moreover
    \begin{enumerate}
        \item Each root of $\overline\gamma_z(X)$ is contained in $\widehat{\calq^z}$ and distinct roots lie in distinct planes of the pencil.
        \item  If $z^2+4z$ is a non-square then $\gamma_z(X)$ has no linear factors.
        \item If $z=-4$ then $\gamma_z(X)$ has the unique linear factor $X-1$.
        \item If $z^2+4z$ is a square then $\gamma_z(X)$ has exactly 2 linear factors $X - r_+$ and $X - r_-$.
  \end{enumerate}
\end{thm}

\begin{proof}
Let $t$ be any root of $\overline\gamma_z(X)$.  Then $I_0(t,1) = \Frob^2(t)$ is also a root of  $\overline\gamma_z(X)$ and thus $\tau_z(t) = \gamma_z(I_0(t,1)) =0$ showing that $t \in \widehat{\calq^z}$.  Consequently the minimal polynomial of $t$ has degree $m=\alpha(z,p)$.
Since $I_0(\calp_i) = t_i$ we see that no two distinct roots of $\overline\gamma_z(X)$ can lie in a single plane $\calp_i$.
\end{proof}

\begin{cor}\label{z=-4}
Let $\legendre{y}{p}$ denote the Legendre symbol.  Then we have that the order of appearance of $z \in \field_p^*$  must divide $\deg(\overline\gamma_z(Z)) = p - \legendre{z^2+4z}{p}$.  In particular, $\alpha(-4,p) = p$ for all $p$ and $\overline\gamma_{-4}(X) = X^p - 1 -2 \sum_{i=1}^{p-1} X^i$ is always irreducible.
\end{cor}

\subsection{Roots of $\gamma_z(X)$ in $\field_p$} \ \\

Let $z\in \FF$ and suppose that $z^2+4z$ is a quadratic residue mod $p$.  Then $\gamma_z(X)$ has the two linear factors
 $X-r_+$ and $X-r_-$ where $r_\pm = \frac{-(z+2)\pm\sqrt{z^2+4z}}{2}$.  Note that $r_+$ and $r_-$ are inverses of one another.
 Furthermore $r_+=r_-$ if and only if $z=-4$ or $z=0$ when $r_+=r_-=1$ and $r_+=r_-=-1$ respectively.

Let $\calp \in\calq^z$ be the plane with $I_0(\calp) = r_+$.  Then for $a \in \widehat{\calp}$ we must have $a^p \in \P$ and thus there exists $c,d \in \field_p$ with $c\neq 0$ such that $a^p = c a + d$.   Now $I_0(a,1) = (a^p -a)^{p-1} = r_+$. Therefore
    \begin{align*}
         r_+ & = \frac{(a^p-a)^p}{a^p-a}
                                      = \frac{((c-1)a + d)^p}{(c-1)a + d}\\
                                     & =  \frac{((c-1)(ca+d) + d}{(c-1)a+d}
                                      = \frac{ c(c-1)a +cd}{(c-1)a+d}
                                     = c
      \end{align*}
Therefore $a^p =  r_+ a + d$.

Suppose now that $z\neq 0, -4$ so that $r_+\neq \pm 1$.  Applying $\Frob^{\ell}$ yields
   $$
        a^{p^\ell} =  r_+^\ell a + (1+r_+ + r_+^2 + \dots + r_+^{\ell-1})d = r_+^\ell a + \left(\frac{r_+^{\ell}-1}{r_+-1}\right)d\ .
    $$
Thus the smallest field containing $a$ is the the field of order $p^\ell$ where $r_+$ has order $\ell$ as an element of  the multiplicative group $\field_p^*$.  In summary we have proved the following theorem.

\begin{thm}  \label{thm:artin connection}
Let $z \in \field$ with $z\neq 0,-4$ be such that $z^2+4z$ is a square. Then the two elements $\frac{-(z+2)\pm\sqrt{z^2+4z}}{2}$ of $\field_p^*$ have multiplicative order $\alpha(z,p)$.
\end{thm}

\begin{rem}
    \begin{enumerate}
        \item A slight modification of the above analysis works for the cases $z=0$ (resp.~$z=-4$) when have $r_+=r_-=-1$ (resp.~$r_+=r_-=1$)  and shows that the minimum $\ell$ with  $\Frob^\ell(a)=a$ is $\ell=p$ (resp. $\ell=2$).
        \item As we will see below (Theorem~\ref{thm:better artin connection}) we may also prove the preceeding result, indeed a stronger version of it, using the theory of Lehmer numbers.
  \end{enumerate}
  \end{rem}
The order of elements of the finite cyclic group $\field_p^*$ has been much studied.  The most famous outstanding question is Artin's Conjecture on primitive roots.  Emil Artin made his conjecture in 1927, \cite{Artin:82}. It states in part the following
\begin{conj}[Artin's Conjecture]
Suppose that $R$ is an integer which is not a perfect square and is not -1.  Then $R$ is a primitive root modulo $p$ for infinitely many primes $p$.
\end{conj}
Here, we say that $R$ is a primitive root modulo $p$ if the residue class of $R$ mod $p$ is a generator of
the multiplicative cyclic group $\FF_p^*$.
Artin included more in his conjecture, giving predicted densities for the set of primes for which $R$ is a primitive root.

In 1984, R.~Gupta and M.R.~Murty, \cite{Gupta+Muarty:84}. showed that there are sets of thirteen integers which include at least one for which Artin's conjecture is true.  This also showed that Artin's conjecture is true for almost all integers $R$.  Later R.~Gupta, V.K.Murty and M.R.~Murty, \cite{Gupta+Murty+Murty:87} reduced the bound of thirteen to seven.  Then in 1986, D.R.~Heath-Brown, \cite{Heath-Brown:86}, further improved this bound showing that there are at most two primes for which Artin's conjecture fails and at most three squarefree integers for which it fails.

Given these results, it is surprising that at present, there is no integer known to be a primitive root mod $p$ for infintely many primes $p$.  We recommend the article \cite{Murty:88} for a more detailed discussion of these results and their history.

 \subsection{Following Lehmer}\ \\

 In 1930, D.H. Lehmer, \cite{Lehmer:30}, extended the theory of Lucas sequences.   He considered two coprime integers $Q$ and $Z$ and the two numbers $A$ and $B$ such that $A+B=\sqrt{Z}$ and $AB=Q$. 

 Put $$
        U_n(\sqrt{Z},Q) 
               =
                                    \begin{cases}
                                        \frac{A^n-B^n}{A-B}, & \text{if $n$ is odd};\\
                                        \frac{A^n-B^n}{A^2-B^2}, & \text{if $n$ is even}.
                                    \end{cases}
    $$
 Then $U_n(\sqrt{Z},Q)$ satisfies
   $$
        U_n(\sqrt{Z},Q) = \begin{cases}
                                    Z U_{n-1}(\sqrt{Z},Q) - Q U_{n-2}(\sqrt{Z},Q), & \text{if $n$ is odd};\\
                                    U_{n-1}(\sqrt{Z},Q) - Q U_{n-2}(\sqrt{Z},Q), & \text{if $n$ is even}.
                                \end{cases}
   $$
Note that if $z \in \field_p$ then $z^{\theta(n-3,p)} =
    \begin{cases}
       z,  & \text{if $n$ is odd};\\
       1, & \text{if $n$ is even}.
    \end{cases}$\\
Thus the restriction of $f_{n,p}$ to $\field_p$ satisfies
    $$
    f_{n,p}(z) = \begin{cases}
                            z f_{n-1,p}(z) + f_{n-2,p}(z), & \text{if $n\geq 3$ is odd};\\
                            f_{n-1,p}(z) + f_{n-2}(z), & \text{if $n\ge 2$ is even},
                        \end{cases}
    $$
where $f_{0,p}(z)=0$ and $f_{1,p}(z)=1$.  These are exactly the recursion relations which define the classical Morgan-Voyce polynomials.  They are also the recursion relations which define $U_n(\sqrt{Z},-1)$. Therefore for any integer $Z$ we have
    $$
        \MV{n}(Z) = U_n(\sqrt{Z},-1) \mbox{ and } f_{n,p}(z) = U_n(\sqrt{Z},-1) \pmod{p}
    $$
where $Z \pmod{p}=z$. Therefore
    $$
    f_{n,p}(z) =  \begin{cases}
                            \frac{a^n-b^n}{a-b}, & \text{if $n$ is odd};\\
                            \frac{a^n-b^n}{a^2-b^2}, & \text{if $n$ is even},
                        \end{cases}
   $$
where $a$ and $b$ are the roots of the polynomial $X^2 - \sqrt{z}X - 1\in\field[X]$ provided $a^2\neq b^2$, that is, provided $z\in\field_p$ with $z \neq 0,-1/4$.

\bigskip
The fact that $\MV{n}(Z) \pmod{p} = f_{n,p}(z)$ implies the following.
\begin{thm}\label{thm:MV apparition}
Let $z \in \field_p^*$ and let $Z$ be an integer representative of $z$.  Then $\alpha(z,p)$ is the order of appearance of $p$ in the integer sequence of Morgan-Voyce values
$$\MV{1}(Z), \MV{2}(Z), \MV{3}(Z),\dots, \MV{p+1}(Z)\ ,$$
 i.e., $\alpha(z,p)$ is the least positive integer $m$ such that $p$ divides $\MV{m}(Z)$.
  \end{thm}
We have shown that for $z \neq 0, -4$ that $f_{n,p}(z)=0$ if and only if $a^n=b^n$ if and only if $(a/b)^n = 1$ where $a,b$ are the roots of  $X^2 - \sqrt{z}X - 1$.  Let $r=a/b$.  Then $r=-a^2$ since $b=-1/a$.   Also $z=(a+b)^2 = (a-1/a)^2 = a^2 - 2 -a^{-2} = -r - 2 -1/r$.

Note that $z=0$ if and only if $\{a,b\}=\{1,-1\}$ if and only if $r=-1$.  Also $z=-4$ if and only if $a=b=1$ if and only if $r=1$.  Thus $z\in\{0,-4\}$ if and only if $r^2=1$ if and only if $a^4=1$.  Thus we have proved the following generalization of Theorem~\ref{thm:artin connection}.
\begin{thm}\label{thm:better artin connection}
Suppose that $z \in \field_p$ and $z \neq 0,-4$.  Put $r = \frac{-z-2\pm\sqrt{z^2+4z}}{2} \in \field_{p^2}$.  The multiplicative order of the $r$ as an element of $\field_{p^2}$ is the order of appearance of $z$.
\end{thm}

Combining our results we have
\begin{thm}\label{artin result}
Let $R \in \bbz$ with $R \neq 0, \pm 1$.   Let $r\in \field_p$ be the residue class of $R$ so that $z=-r-2-r^{-1}\in \field_p$ is the residue class of $-R-2-R^{-1}$.  Then $\gamma_z(X) = X^{p+1} + (z+1)X^p + X + 1 \in \field_p[X]$ has two linear factors and all its other irreducible factors have the same degree $m$ where $m$ is the order of $r$ as an element of the group $\field_p^*$.  In particular, $R$ is a primitive root mod $p$ if and only if $\gamma_z(X)/(X^2+(z+2)X+1)\in \field_p[X]$ is irreducible.
\end{thm}

In general, given $r \in \field_p\setminus \set{0,\pm 1}$ we consider the set map $\sigma(r) = -r-2-r^{-1}$.  Then  $\im(\sigma) \subset \field_p$ is the set of elements $z \in \field_p$ such that $z^2+4z$ is a non-zero square.  We have $\sigma(r) = \sigma(r^{-1})$ so that $\sigma$ is a $2$ to $1$ map. Hence
    $$
        \#\im(\sigma) = (p-3)/2\ .
    $$
We also recall that the number of elements of order $k$ in the cyclic group $C_n$ is given by the Euler function $\phi(k)$ for $k$ dividing $n$.  We recall also that $\gamma_0(X) = X^{p+1}+X^p+X+1 = (X+1)^{p+1}$, for all $p$.   We note, however, that for $r = \pm 1$, where we have $\sigma(1) = -4$ and $\sigma(-1) = 0$. By \corref{z=-4} we know that  when $r=-1 = r^{-1}$, we have $z=-4$ and $\gamma_{-4}(X)$ has $1$ linear factor and $1$ irreducible factor of degree $p$.  Hence $\alpha(-4,p)=p$.

\begin{exa}\label{p=19}
We consider here the prime $19$ and the set
    $$
        \im(\sigma) = \set{1,5,7, 8,10,14,16,18}\ ,
    $$
with $8$ elements.
    \begin{enumerate}
        \item The cyclic group $\field_{19}^*$ has $\phi(18)=6$ primitive elements.  These come in the pairs
        $(r,1/r)$ = $(2,10), (3,13), (14,15)$ associated to the values $5$, $1$, and $7$ respectively.  For these three values $\gamma_z(X)$ has $2$ linear factors and one irreducible factor of degree $18$.
        Hence $\alpha(1,19)=\alpha(5,19)=\alpha(7,19)=18$.
        \item $\field_{19}^*$ has $\phi(9) = 6$ elements of order $9$, $(r,r^{-1})=(4,5), (6,16), (9,17)$ associated to the values $4,14,10$ respectively.
          For these three values of $z$ we have that $\gamma_{z}(X)$ has $2$ linear factors and $2$ irreducible factors of degree $9$.
            Hence $\alpha(4,19)=\alpha(10,19)=\alpha(14,19)=9$.
        \item $\field_{19}^*$ has $\phi(6) = 2$ elements of order $6$, $(r,r^{-1})=(8,12)$ associated to the value $16$.  We have that $\gamma_{16}(X)$ has $2$ linear factors and $3$ irreducible factors of degree $6$.  Hence $\alpha(16,19)=6$.
        \item $\field_{19}^*$ has $\phi(3)=2$ elements of order $3$ , $(r,r^{-1}) = (7,11)$ associated to the value $18$.  We have that $\gamma_{18}(X)$ has $2$ linear factors and $6$ irreducible factors of degree $3$.  Hence $\alpha(18,19)=\alpha(-1,19)=3$.
    \end{enumerate}
\end{exa}

\section{Trinomials}
We have $\gamma_z(X) = X^{p+1} +(z+1)X^p + X + 1$.  We may use Descartes' transformation to eliminate the term of degree $p$.  Here this is done by evaluating $\gamma_z(X-z-1)$ as we did in the proof of Theorem~\ref{thm:Frob2}.  Putting $\beta_z(X) := \gamma_z(X-z-1)$ we obtain the trinomial $\beta_z(X) = X^{p+1} - zX - z$.  Of course the degrees of the irreducible factors of $\beta_z(X)$ are the same as those for $\gamma_z(X)$.  We may further simplify by using $\delta_z(X) := z^{-2}\beta_z(zX) = X^{p+1} - X - z^{-1}$.   Again the degrees of the irreducible factors of $\delta_z(X)$ are the same as those for $\beta_z(X)$ and $\gamma_z(X)$.

The problem of factoring trinomials over finite fields has been considered by many authors. In particular trinomials of the form $X^{p^k+1} - aX -b$ were carefully studied by S.~Agou, \cite{Agou-IrreduciblePolys:84, Agou-Hyponormal:81, Agou:77}.  Agou derived intricate algebraic conditions to explicitly determine the degrees of the irreducible factors of such polynomials.  These conditions can be shown to be related to the polynomials $f_{m,p}(X)$.  While Agou's work is sometimes overlooked in the literature, other authors have also considered such trinomials, Bluher, \cite{Bluher:04}, Coulter and Henderson, \cite{Coulter+Henderson:04},  Kim and Mesnager, \cite{Kim+Mesnager:20}, and Stichtenoth and Topuzou\u{g}lu, \cite{Stichtenoth+Topuzouglu:12}.
\begin{thm}
Consider the trinomial $X^{p+1}-aX-b\in \FF_p[X]$ with $a \neq 0$ and put $\zeta := b/a^2$.  Let $r$ be a root of $\zeta X^2 +(2\zeta+1)X + \zeta$ and let $m$ denote the multiplicative order of $r$ as an element of $\field_{p^2}$.  Then the multiset of degrees of the irreducible factors of $X^{p+1}-aX-b$ is
    \begin{enumerate}
        \item $(m,m,\dots,m)$ with $m \geq 3$, if $1+4\zeta$ is a non-square in $\field_p$;
        \item $(1,1,m,m,\dots,m)$ with $m \geq 3$, if $1+4 \zeta$ is a non-zero square;
        \item $(1,p)$ if $\zeta=-1/4$; \label{case3,z=-4}
        \item $(1,1,\dots,1)$ if $\zeta = 0$.\label{case4, z=0}
        \end{enumerate}
\end{thm}
\begin{proof}
If $b=0$ then $X^{p+1}-aX = X^{p+1} -a^p X = X(X^p-a^p) = X(X-a)^p$.   Suppose then that $b\neq 0$ and thus $\zeta \neq 0$.  Evaluating at $aX$ and dividing by $a^2$ we have $\big((aX)^{p+1} - a(aX) - b\big)/a^2 = X^{p+1} - X - b/a^2 = \delta_{\zeta^{-1}}(X)$.  Thus the factorization of $X^{p+1}-aX-b$ has the same form as the factorization of $\beta_{\zeta^{-1}}(X)$.  Put $z=\zeta^{-1}$.  The form of the factorization of $\beta_z(X)$ into irreducibles is governed by whether $z^2+4z$ is a square and by the multiplicative order $m$ of $r$ where $r$ is a root of $X^2+(z+2)X+1$.
  Since  $z^2+4z=\zeta^{-2}+4\zeta^{-1} = (1+4\zeta) / \zeta^2$ and
   $X^2+(z+2)X+1 = X^2 +(2+\zeta^{-1})X+1 = \big(\zeta X^2 + (2\zeta+1)X + \zeta\big)/\zeta$ the result follows.
   Note that $m=1$ corresponds to $z=-4$ which is case~(\ref{case3,z=-4}) and $m=2$ corresponds to $z=0=b$  which is case~(\ref{case4, z=0}).
\end{proof}

\begin{rem}
We may invert our point of view and use our results to obtain irreducible polynomials in $\field_p[X]$ of certain specified degrees.  Let $m$ be a positive integer which divides $(p^2-1)/2$.  Then $m$ divides either $p-1$ or $p+1$.   Choose $r \in \field_{p^2}$ of order $m$ and define $z := -r-2-1/r$.
By Theorem~\ref{artin result}, every irreducible factor of $\overline{\gamma}_z(X)$ has degree $m$.
\end{rem}

 Most of the results proved in this article may be extended to the case where $z \notin \field_p$.  These generalizations will be discussed in \cite{Campbell+Wehlau:20}.


\section{Primes with large Fibonacci orders of appearance}

To study classical orders of apparition in the Fibonacci sequence we focus on $z=1$ and $r$ which satisfies $-r-2-1/r = 1$.  Thus $r$ and its inverse are the two numbers $r_{\pm} = \frac{-3\pm\sqrt{5}}{2}$.

The set of primes $p$ which satisfy $\alpha(p) = p+1$ is of some interest.  The list of the first few primes which satisfy this condition are indexed as sequence A000057 in the OEIS,  \cite{OEIS}.  As observed by Cubre and Rouse, \cite{Cubre+Rouse:14},  it is not yet known if there are infinitely many primes $p$ for which $\alpha(p) = p+1$. For $\alpha(p)=p+1$ we must have that $z^2+4z=5$ is a non-square, i.e., $p \equiv \pm 2 \pmod{5}$.  From  Theorem~\ref{artin result}  we see that $\alpha(p) = p+1$ if and only if  the two numbers $r_{\pm} \in\FF_{p^2}$ each have order $p+1$.

Consider now the set of primes $p$ for which $\alpha(p) = p-1$.  These are the primes for which $p \equiv \pm 1 \pmod{5}$ and the two numbers $r_\pm$ are both primitive roots mod $p$.  By (a generalized version of) Artin's conjecture this happens infinitely often which would imply that this set of primes is infinite.  Shanks and Taylor discovered this connection already in their study of so-called Fibonacci primitive roots, see \cite{Shanks:72, Shanks+Taylor:73}.  They also derived from Artin's conjecture a conjectural density, among all primes, of 0.177135... for the primes $p$ which satisfy $\alpha(p)=p-1$. However it remains unknown whether this set of primes is infinite.

\appendix
\section{Proving the recusive formula}
In this appendix we give a proof of Proposition~\ref{recursion}.  There is nothing deep in the proof which uses only the properties of bracket polynomials.  However, the computations are somewhat involved.

Recall that $\nu = -\frac{[0,2][1,3]}{[0,1][2,3]}$.

\begin{lem}
 $$  \nu^{\theta(2k-1,p)} = \frac{[1,2][2k,2k+2][2k+1,2k+2]}{[0,1][0,2][2k,2k+1]}[0,1]^{-2\theta(2k+1,p)}\text{ for }k\geq 1.$$
\end{lem}

\begin{proof}
When $k=1$ we have
    \begin{align*}
       &\frac{[1,2][2,4][3,4]}{[0,1][0,2][2,3]}[0,1]^{-2\theta(3,p)} = \frac{[2,4]}{[0,2]}\frac{[0,1]^{p+p^3}}{[0,1]^{1+p^2}} [0,1]^{-2\theta(3,p)}\\
       &= \frac{[2,4]}{[0,2]} [0,1]^{-p^3+p^2-p+1} = \frac{[2,4]}{[0,2]} [0,1]^{-\theta(3,p)}\\
      \text{Con}&\text{versely}\\
      \nu^{\theta(1,p)} & = \left(\frac{[0,2][1,3]}{[0,1][2,3]}\right)^{p-1} = \frac{[1,3][2,4]}{[1,2][3,4]} \frac{[0,1][2,3]}{[0,2][1,3]}\\
      &=\frac{[2,4]}{[0,2]} \frac{[0,1]^{1+p^2}}{[0,1]^{p+p^3}} = \frac{[2,4]}{[0,2]} [0,1]^{-\theta(3,p)} \ .
  \end{align*}

For $k \geq 2$ we have
    \begin{align*}
        \nu&{}^{\theta(2k,p)} = \nu^{\theta(2k-2,p)} \nu^{p^{2k}-p^{2k-1}}\\
        &= -\frac{[0,1][0,2][2k-1,2k+1]}{[1,2][2k-1,2k][2k,2k+1]}[0,1]^{2\theta(2k-1,p)}\frac{\nu^{p^{2k}}}{\nu^{p^{2k-1}}}\\
        &= -\frac{[0,1][0,2][2k-1,2k+1]}{[1,2][2k-1,2k][2k,2k+1]} [0,1]^{2\theta(2k-1,p)}\\
        &\qquad\cdot\frac{[2k,2k+2][2k+1,2k+3]}{[2k,2k+1][2k+2,2k+3]}
                               \frac{[2k-1,2k][2k+1,2k+2]}{[2k-1,2k+1][2k,2k+2]} \\
       &= -\frac{[0,1][0,2]}{[1,2][2k+1,2k+2]}\frac{[2k+1,2k+3][2k+1,2k+2]^2}{[2k,2k+1]^2[2k+2,2k+3]} [0,1]^{2\theta(2k-1,p)}\\
       &= -\frac{[0,1][0,2][2k+1,2k+3]}{[1,2][2k,2k+1][2k+2,2k+3]}\frac{[2k+1,2k+2]^2}{[2k,2k+1]^2} [0,1]^{2\theta(2k-1,p)}\\
       &= -\frac{[0,1][0,2][2k+1,2k+3]}{[1,2][2k,2k+1][2k+2,2k+3]}\left(\frac{[0,1]^{p^{2k+1}}}{[0,1]^{p^{2k}}}\right)^2 [0,1]^{2\theta(2k-1,p)}\\
       &= -\frac{[0,1][0,2][2k+1,2k+3]}{[1,2][2k,2k+1][2k+2,2k+3]}\left( [0,1]^{p^{2k+1}-p^{2k}+\theta(2k-1,p)}\right)^2\\
        &= -\frac{[0,1][0,2][2k+1,2k+3]}{[1,2][2k,2k+1][2k+2,2k+3]} [0,1]^{2\theta(2k+1,p)}
    \end{align*}
\end{proof}

\begin{lem}
   $$
        \nu^{\theta(2k,p)} = -\frac{[0,1][0,2][2k+1,2k+3]}{[1,2][2k+1,2k+2][2k+2,2k+3]}[0,1]^{2\theta(2k+1,p)}\text{ for }k\geq 1\ .
   $$
\end{lem}
\begin{proof}
When $k=1$
    \begin{align*}
        \nu^{\theta(2,p)} & = \left(\frac{[0,2][1,3]}{[0,1][2,3]}\right)^{p^2-p+1}
         = \frac{[2,4][3,5]}{[2,3][4,5]} \frac{[1,2][3,4]}{[1,3][2,4} \frac{[0,2][1,3]}{[0,1][2,3]}\\
        &= \frac{[0,2][3,5][1,2][3,4]}{[0,1][2,3][2,3][4,5]}
        = \frac{[0,1][0,2][3,5]}{[1,2][3,4][4,5]}\left(\frac{[1,2][3,4]}{[0,1][2,3]} \right)^2\\
        &=  \frac{[0,1][0,2][3,5]}{[1,2][3,4][4,5]} [0,1]^{2p + 2p^3 - 2 - 2p^2}
        =  \frac{[0,1][0,2][3,5]}{[1,2][3,4][4,5]} [0,1]^{2\theta(3,p)}\ .
    \end{align*}

For $k \geq 2$ we have
    \begin{align*}
        \nu&{}^{\theta(2k+1,p)} = \nu^{\theta(2k-1,p)} \nu^{p^{2k+1}-p^{2k}}\\
        &= \frac{[1,2][2k,2k+2][2k+1,2k+2]}{[0,1][0,2][2k,2k+1]}[0,1]^{-2\theta(2k+1,p)}\frac{\nu^{p^{2k+1}}}{\nu^{p^{2k}}}\\
        &= \frac{[1,2][2k,2k+2][2k+1,2k+2]}{[0,1][0,2][2k,2k+1]}[0,1]^{-2\theta(2k+1,p)}\\
        &\qquad\cdot\frac{[2k+1,2k+3][2k+2,2k+4]}{[2k+1,2k+2][2k+3,2k+4]} \frac{[2k,2k+1][2k+2,2k+3]}{[2k,2k+2][2k+1,2k+3]} \\
        &= \frac{[1,2][2k+2,2k+4][2k+2,2k+3]}{[0,1][0,2][2k+3,2k+4]}[0,1]^{-2\theta(2k+1,p)} \\
       &= \frac{[1,2][2k+2,2k+4][2k+3,2k+4]}{[0,1][0,2][2k+2,2k+3]}\left(\frac{[2k+2,2k+3]}{[2k+3,2k+4]}\right)^2 [0,1]^{-2\theta(2k+1,p)} \\
       &= \frac{[1,2][2k+2,2k+4][2k+3,2k+4]}{[0,1][0,2][2k+2,2k+3]}\left(\frac{[0,1]^{p^{2k+2}}}{[0,1]^{p^{2k+3}}}\right)^2 [0,1]^{-2\theta(2k+1,p)} \\
       &= \frac{[1,2][2k+2,2k+4][2k+3,2k+4]}{[0,1][0,2][2k+2,2k+3]}\left([0,1]^{p^{2k+2}-p^{2k+3}-\theta(2k+1,p)}\right)^2  \\
       &= \frac{[1,2][2k+2,2k+4][2k+3,2k+4]}{[0,1][0,2][2k+2,2k+3]} [0,1]^{-2\theta(2k+3,p)}
    \end{align*}
\end{proof}
Define
    \begin{align*}
        F_{2k+1,p} &:= (-1)^{k}[0,2k+1] [0,1]^{-\theta(2k,p)}\quad\text{for }k\geq 1 \ \ \text{ and}\\
        F_{2k,p} & := (-1)^{k+1}\frac{[1,2]}{[0,1][0,2]}[0,2k] [0,1]^{-\theta(2k-1,p)} \quad\text{for }k\geq 1 \ .
     \end{align*}
We recall Proposition~\ref{recursion}: \recursionProp*
\begin{proof}
  The proof is by induction on $m$.  For the base we consider $m=1,2,3$ separately.
    \begin{align*}
        F_{1,p} :=& [0,1][0,1]^{-\theta(0,p)} = [0,1][0,1]^{-1} = 1.\\
        F_{2,p} :=& \frac{[1,2][0,2]}{[0,1][0,2]}[0,1]^{-\theta(1,p)} = \frac{[0,1]^p}{[0,1]}[0,1]^{-(p-1)} = 1.\\
        F_{3,p} :=& - [0,3] [0,1]^{-\theta(2,p)} = -[0,3] [0,1]^{-p^2+p-1}\\
        =& \frac{-[0,3][1,2]}{[0,1][2,3]}= \frac{-[0,2][1,3] + [0,1][2,3]}{[0,1][2,3]} = \nu + 1 = \nu^{\theta(0,p)} F_{2,p} + F_{1,p}.
    \end{align*}

Now the induction step.  We treat even and odd values of $m$ separately.
    \begin{align*}
        \nu&^{\theta(2k-1,p)} F_{2k+1,p} + F_{2k,p}\\
        & = \nu^{\theta(2k-1,p)} (-1)^k[0,2k+1] [0,1]^{-\theta(2k,p)} + (-1)^{k+1}\frac{[1,2][0,2k]}{[0,1][0,2]} [0,1]^{-\theta(2k-1,p)}\\
        &=\left(\frac{[1,2][2k,2k+2][2k+1,2k+2]}{[0,1][0,2][2k,2k+1]}[0,1]^{-2\theta(2k+1,p)} \right)\\
        &\qquad\cdot \left((-1)^{k}\frac{[0,2k+1]}{[2k+1,2k+2]} [0,1]^{\theta(2k+1,p)}\right) + (-1)^{k+1}\frac{[1,2][0,2k]}{[0,1][0,2]} [0,1]^{-\theta(2k-1,p)}\\
        &= (-1)^{k}\frac{[1,2][2k,2k+2][0,2k+1]}{[0,1][0,2][2k,2k+1]}[0,1]^{-\theta(2k+1,p)}\\
        & \qquad\qquad+ (-1)^{k+1}\frac{[1,2]}{[0,1][0,2]}\frac{[0,2k][2k+1,2k+2]}{[2k,2k+1]} [0,1]^{-\theta(2k+1,p)}\\
        &= \frac{(-1)^k([2k,2k+2][0,2k+1] -[0,2k][2k+1,2k+2])}{[2k,2k+1]} \frac{[1,2]}{[0,1][0,2]}[0,1]^{-\theta(2k+1,p)} \\
        &=  \frac{(-1)^{k}[0,2k+2][2k,2k +1]} {[2k,2k+1]} \frac{[1,2]}{[0,1][0,2]}[0,1]^{-\theta(2k+1,p)}\\
        &= (-1)^{k}\frac{[1,2][0,2k+2]}{[0,1][0,2]}[0,1]^{-\theta(2k+1,p)} \\
        &= F_{2k+2,p}
    \end{align*}
Finally the induction step for general $m$ of the other parity.
    \begin{align*}
        \nu&^{\theta(2k,p)} F_{2k+2,p} + F_{2k+1,p} = \left(-\frac{[0,1][0,2][2k+1,2k+3]}{[1,2][2k+1,2k+2][2k+2,2k+3]}[0,1]^{2\theta(2k+1,p)} \right)\\
           &\qquad\cdot \left((-1)^k\frac{[1,2][0,2k+2]}{[0,1][0,2]} [0,1]^{-\theta(2k+1,p)}\right)+ (-1)^k[0,2k+1] [0,1]^{-\theta(2k,p)}\\
           &= -(-1)^k\frac{[2k+1,2k+3][0,2k+2]}{[2k+1,2k+2][2k+2,2k+3]}[0,1]^{\theta(2k+1,p)}\\
           & \qquad\qquad+ (-1)^k\frac{[0,2k+1][2k+2,2k+3]}{[2k+1,2k+2][2k+2,2k+3]} [0,1]^{\theta(2k+1,p)}\\
           &= \frac{(-1)^{k}(-[2k+1,2k+3][0,2k+2] + [0,2k+1][2k+2,2k+3])}{[2k+1,2k+2][2k+2,2k+3]} [0,1]^{\theta(2k+1,p)} \\
           &=  \frac{(-1)^{k+1}[0,2k+3][2k+1,2k +2]} {[2k+1,2k+2][2k+2,2k+3]} [0,1]^{\theta(2k+1,p)}\\
           &=  (-1)^{k+1}\frac{[0,2k+3][2k+3,2k+4]} {[2k+2,2k+3][2k+3,2k+4]}[0,1]^{\theta(2k+1,p)}\\
           &= (-1)^{k+1}\frac{[0,2k+3]}{[2k+2,2k+3]}[0,1]^{\theta(2k+1,p)} \\
           &= (-1)^{k+1}[0,2k+3][0,1]^{p^{-(2k+2)+\theta(2k+1,p)}} \\
           &= (-1)^{k+1}[0,2k+3][0,1]^{-\theta(2k+2,p)}\\
           &= F_{2k+3,p}
    \end{align*}
\end{proof}

\section*{Acknowledgments.}
This research is supported in part by the Natural Sciences and
Engineering Research Council of Canada. The symbolic computation
language MAGMA (\url{http://magma.maths.usyd.edu.au/}) was very helpful.

\bibliographystyle{amsplain}
\def\cprime{$'$}
\providecommand{\bysame}{\leavevmode\hbox to3em{\hrulefill}\thinspace}
\providecommand{\MR}{\relax\ifhmode\unskip\space\fi MR }
\providecommand{\MRhref}[2]{%
  \href{http://www.ams.org/mathscinet-getitem?mr=#1}{#2}
}
\providecommand{\href}[2]{#2}

\end{document}